\documentclass[11pt]{article}
\usepackage{a4wide}
\usepackage{amsmath,amssymb,amsthm}

\providecommand{\R}{\mathbb{R}}
\providecommand{\N}{\mathbb{N}}

\providecommand{\ee}{\varepsilon}
\providecommand{\BB}{\mathcal{B}}
\providecommand{\MM}{\mathcal{M}}
\providecommand{\PP}{\mathcal{P}}
\providecommand{\LL}{\mathcal{L}}
\providecommand{\FF}{\mathcal{F}}
\providecommand{\QQ}{\mathcal{Q}}

\providecommand{\lin}{\mathop{\rm span}\nolimits}

\providecommand{\pos}{\mathop{\rm pos}\nolimits}

\newtheorem{Thm}{Theorem}
\newtheorem{Lem}[Thm]{Lemma}
\newtheorem{Prop}[Thm]{Proposition}
\newtheorem{Cor}[Thm]{Corollary}

\theoremstyle{definition}
\newtheorem{Def}[Thm]{Definition}
\theoremstyle{remark}

\title{The Yao-Yao equipartition of a measure}
\author{Joseph Lehec\footnote{
LAMA (UMR CNRS 8050) Universit\'e Paris-Est
}}
\date{September 2008}

\begin{document}
\maketitle

\begin{abstract} 
The Yao-Yao partition theorem states that for any probability measure $\mu$ on $\mathbb{R}^n$ having a density which is continuous and bounded away from $0$,  it is possible to partition $\mathbb{R}^n$ into $2^n$ regions of equal measure for $\mu$ in such a way that every affine hyperplane of $\mathbb{R}^n$ avoids at least one of the regions. We 
give a constructive proof of this result and extend it to slightly more general measures. 
\bigskip

\noindent
Published in Arch. Math. 92 (4) (2009) 366--376.
\end{abstract}

\section{Introduction}
In \cite{yaoyao}, Yao and Yao show that for any probability measure $\mu$ on $\mathbb{R}^n$ having a density which is continuous and bounded away from $0$,  it is possible to partition $\mathbb{R}^n$ into $2^n$ regions of equal measure for $\mu$ in such a way that every affine hyperplane of $\mathbb{R}^n$ avoids at least one of the regions. This theorem was designed for computational geometry purposes but it turned out to be useful 
in other areas of mathematics. For instance, the authors of~\cite{alon} use it to prove 
Ramsey type theorems and in~\cite{moi}, we show a connection with the Blaschke-Santal\'o
inequality. More precisely, using a partition \`a la Yao and Yao, we reduce a general 
functional form of the Blaschke-Santal\'o inequality (due to Fradelizi and Meyer~\cite{fm}) 
to an easy inequality between functions deÞned on $\mathbb{R}_+$. 
The proof of Yao and Yao is by induction on the dimension and uses the Borsuk-Ulam theorem. The purpose of this paper is to show that the Yao-Yao theorem 
can be obtained in a much more concrete way, by applying the real intermediate 
values theorem again and again. Of course this proof is longer, but we believe that 
it gives a better understanding of the structure of the Yao-Yao partition. Also we 
are able to get rid of the technical assumptions ($\mu$ having a continuous density 
bounded away from 0) which are annoying for applications. The article \cite{yaoyao} being 
very sketchy, we intend to give (almost) every detail.

The paper deals with finite dimensional real affine spaces; if $E$ is such a space, $\vec{E}$ denotes the associated vector space. We say that $\PP$ is a \emph{partition} of $E$ if $\cup \PP = E$ and if the interiors of two distinct elements of $\PP$ do not intersect.
For instance, with this definition, the set $\PP=\{(-\infty,a],[a,+\infty)\}$ is a partition of $\R$.
\begin{Def}
\label{5YYP}
If $E=\{x\}$ is an affine space of dimension $0$, we say that $\PP$ is a Yao-Yao partition of $E$ if $\PP=\{x\}$ and we define the center of $\PP$ to be $x$. \\
Let $E$ be an affine space of dimension $n\geq 1$. We say that $\PP$ is a Yao-Yao partition of $E$ if there exists an affine hyperplane $F$ of $E$, a vector $v\in \vec{E} \backslash \vec{F}$ and two Yao-Yao partitions $\PP_1$ and $\PP_{-1}$ of $F$ \emph{having the same center $x$} such that 
\[ \PP= \bigl\{  A + \R_- v {\ \vert \ } A\in \PP_{-1} \bigr\} \cup \bigl\{ A + \R_+ v {\ \vert \ } A\in \PP_1 \Bigr\} ,  \]
and we say that $x$ is the center of $\PP$. 
\end{Def}
%
If $E$ has dimension $n$, then a Yao-Yao partition of $E$ has $2^n$ elements and we shall see in the following section that every hyperplane of $E$ avoids at least one of the elements of a Yao-Yao partition. Let us state our main theorem. \\
Let $\MM(E)$ be the set of non-negative Borel measures $\mu$ on $E$ which are finite and satisfy $\mu(F)=0$ for any affine hyperplane $F$. 
\begin{Def}
\label{5equipartition}
Let $\mu\in\MM(E)$, a Yao-Yao \emph{equipartition} $\PP$ for $\mu$ is a Yao-Yao partition of $E$ satisfying 
\begin{equation}
\label{5equip}
\forall A\in\PP, \quad  \mu(A) = 2^{-n} \mu(E).
\end{equation}
We say that $x\in E$ is a Yao-Yao center of $\mu$ if $x$ is the center of a Yao-Yao equipartition for $\mu$. 
\end{Def}
\begin{Thm}
\label{5main}
Let $\mu \in \MM(E)$, there exists a Yao-Yao equipartition for $\mu$. Moreover, if $\mu$ has a center of symmetry $x\in E$, then $x$ is a Yao-Yao center for $\mu$.
\end{Thm}
\section{Main properties}
If $A$ is a subset of $\vec{E}$ we denote by $\pos(A)$ the positive hull of $A$, that is to say the smallest convex cone containing $A$. \\
A Yao-Yao partition $\PP$ of an $n$-dimensional space $E$ has $2^n$ elements and for each $A$ in $\PP$ there exists a basis $v_1,\dotsc,v_n$ of $\vec{E}$ such that 
\begin{equation}
\label{5pos} 
A = x + \pos(v_1,\dotsc,v_n) ,
\end{equation}
where $x$ is the center of $\PP$.
Indeed, assume that $\PP$ is defined by $F,v,\PP_1$ and $\PP_{-1}$ (see Definition~\ref{5YYP}). Let $A\in \PP_1$ and assume inductively that there is a basis $v_1,\dotsc,v_{n-1}$ of $\vec{F}$ such that $A = x  + \pos ( v_1 , \dotsc, v_{n-1} )$. Then $A + \R_+ v =  x + \pos( v , v_1, \dotsc , v_{n-1} )$.
\begin{Prop}
\label{5halfspace}
Let $\PP$ be a Yao-Yao partition of $E$ and $x$ be its center. Any affine half-space containing $x$ contains an element of $\PP$.
\end{Prop}
\begin{proof}
When $E$ has dimension $0$, the result is obvious. Let $E$ have dimension $n\geq 1$ and assume that the proposition holds for any affine space of dimension $n-1$. Let $\ell$ be an affine form on $E$ such that $\ell(x)\geq 0$, and let $H = \{y\in E \,|\, \ell(y)\geq 0\}$. We use the notations of Definition~\ref{5YYP}. By the induction assumption, there exists $A_+ \in \PP_1$ and $A_-\in\PP_{-1}$ such that
\[ \forall y\in A_+,\; \ell(y) \geq 0 \quad
\text{and} \quad \forall y\in A_-,\; \ell(y) \geq 0. \] 
Let $\vec{\ell}$ be the linear form on $\vec{E}$ associated to $\ell$, if $\vec{\ell}(v) \geq 0$ then $\ell(x+tv)\geq 0$ for all $x\in A_+$ and $t\in\R_+$, thus $A_+ + \R_+ v \subset H$. Similarly if $\vec{\ell}(v) \leq 0$ then $A_- + \R_- v \subset H$. In both cases $H$ contains an element of $\PP$.
\end{proof}
\begin{Prop}
\label{5convex}
Let $\mu\in\MM(E)$ and let $K$ be a convex subset of $E$ satisfying $\mu( E\backslash K) < 2^{-n} \mu(E)$. Then any Yao-Yao center $x$ of $\mu$ is contained in $K$.
\end{Prop}
\begin{proof}
Assume on the contrary that there is a center $x$ of $\mu$ outside $K$, and let $\PP$ be an equipartition with center $x$. By Hahn-Banach there is a half-space containing $x$ and disjoint from $K$. By Proposition~\ref{5halfspace}, this half-space contains an element $A$ of $\PP$. So on the one hand $\mu(A) = 2^{-n} \mu(E)$ and on the other hand $A\subset E\backslash K$, thus we get a contradiction.
\end{proof}
\section{Center with respect to a basis}
Let $E$ be an affine space of dimension $n$ and $\LL=(\ell_1,\dotsc,\ell_n)$ be a family of affine forms on $E$ such that the map
\begin{equation*}
x \in E \mapsto \bigl( \ell_1(x),\dotsc,\ell_n(x) \bigr) \in \R^n
\end{equation*}
is one to one. We say that $\LL$ is a \emph{system of coordinates}. If $x\in E$ and $v \in \vec{E}$ we write $x_i$ and $v_i$ for $\ell_i(x)$ and $\vec{\ell}_i(v)$, respectively. We also let $e^1,\dotsc,e^n$ be the basis of $\vec{E}$ satisfying $e^i_j=\delta_{ij}$. 
Let us introduce a restricted notion of Yao-Yao partition.
\begin{Def}
\label{5YYP2}
A Yao-Yao partition $\PP$ of $E$ given by $F,v,x,\PP_1$ and $\PP_{-1}$ is \emph{adapted to $\LL$} if $F=\{ y\in E \, | \, y_1 = x_1 \}$ and if $\PP_1$ and $\PP_{-1}$ are adapted to $({\ell_2}_{|F},\dotsc, {\ell_n}_{|F})$ (which is a system of coordinates on $F$).
\end{Def}
In the sequel, the vector $v$ is called the \emph{axis} of $\PP$, and it is said to be \emph{normalized} when $v_1=1$. Since $\PP$ has $2^n$ elements, there is a one to one map between $\PP$ and the discrete cube $\{-1,1\}^n$. Let us construct such a map.
\begin{Def}
Let $\PP$ be a Yao-Yao partition of $E$ adapted to $\LL$, defined by $x,\PP_1,\PP_{-1}$ and $v$ (with $v_1=1$). Let $\PP(\emptyset) = E$ and let $(\ee_1,\dotsc,\ee_k)$ be a sequence of $\pm 1$ of size $k\in\{1,\dotsc ,n\}$. Recall that $\PP_{\ee_1}$ is a partition of $F$, hence the notation $\PP_{\ee_1}( \cdot)$ is relative to $F$, for instance $\PP_{\ee_1}(\emptyset) = F$. We assume inductively that we have defined $\PP_{\ee_1} (\ee_2,\dotsc,\ee_k)$ and we let
\[\PP(\ee_1,\dotsc,\ee_k) = \PP_{\ee_1}( \ee_2,\dotsc,\ee_k ) + \R_+ (\ee_1 v).\]
\end{Def}
An easy induction shows that 
\begin{equation}
 \label{5PP}
\PP= \bigl\{ \PP( \ee ) {\ \vert \ } \ee\in \{-1,1\}^n  \bigr\} .
\end{equation}
We now give some basic properties of the sets $\PP(\cdot)$.
It is easy to prove by induction that if $\PP$ is a Yao-Yao partition of $E$ adapted to the basis $\LL$, then for all $\pm 1$ sequence $\ee$ of size $k$, there exists a sequence of vectors $v^{1},\dotsc,v^{k}$ satisfying $v^{i}_j=0$ and $v^{i}_i= 1$ for all $j < i \leq k$ (we call \emph{sub-diagonal} such a sequence hereafter) such that
\begin{equation}
 \label{5repres}
\PP (\ee) = x + \pos ( \ee_1 v^1 , \dotsc , \ee_k v^k ) + \lin ( e^{k+1},\dotsc , e^n ) ,
\end{equation}
where $x$ is the center of $\PP$. Besides, the first vector $v^1$ is equal to the axis of $\PP$ (in particular, it does not depend on $\ee$). Let $x,y\in E$ and $v^1,\dotsc,v^n$ and $w^1,\dotsc,w^n$ be two sub-diagonal sequences of $\vec{E}$. Observe that if
\[ x + \pos (v^1,\dotsc,v^n)= y + \pos (w^1,\dotsc,w^n)  \]
then $x=y$ and $v^i=w^i$ for all $i=1,\dotsc,n$.
Let $\PP$ and $\QQ$ be Yao-Yao partitions of $E$ satisfying $\PP(\ee) = \QQ(\ee)$ for some $\pm 1$ sequence $\ee$ of length $k\leq n$. Let $P$ be the projection of $E$ with range $\lin(e^1,\dots,e^k)$ and kernel $\lin(e^{k+1},\dotsc,e^n)$. Using \eqref{5repres}, the equality $P\bigl( \PP(\ee) \bigr)=  P\bigl( \QQ(\ee) \bigr)$ and the observation above, we get
\begin{equation}
\label{5center-axis}
\begin{split}
(x_1, \dotsc, x_ k) & =(y_1,\dotsc, y_k) \quad \text{and} \\
(v_1,\dotsc, v_k) & =(w_1,\dotsc, w_k),
\end{split}
\end{equation}
where $x$ and $y$ are the centers of $\PP$ and $\QQ$, respectively; and $v$ and $w$ are their normalized ($v_1=w_1=1$) axes.

If $S$ is a set of affine functions on $E$ (possibly empty), we let $\BB_E(S)$ be the smallest $\sigma$-algebra of subsets of $E$ making all elements of $S$ measurable. When a set $A$ belongs to $\BB_E( S)$ for some $S$, we say that $A$ depends only on $S$.
Another simple consequence of \eqref{5repres} is that for any $\ee$ of length $k$ the set $\PP(\ee)$ depends only on $(\ell_1,\dotsc,\ell_k)$. \\
Let $\alpha\in\R$ and $F=\{\ell_1=\alpha\}$. For any $u\in\vec{E}\backslash \vec{F}$ we let $\pi_{F,u}: F + \R_+ u \to F$ be the projection which, for any $x\in F$ and $t\in \R_+$, maps $x+ t u$ to $x$. Let $A\subset F$, we let
\begin{equation}
\label{5proj} 
A_u = ( \pi_{F,u} )^{-1} (A) = \{ x\in E \, | \, x_1\geq \alpha \; \text{and} \; x -(x_1-\alpha) \frac{u}{u_1} \in A \} .
\end{equation}
Let $\PP$ be a partition of $E$. Let $k<n$ and $\ee_1,\dotsc,\ee_k$ be a $\pm 1$ sequence. There exists an affine form $\ell\in \lin(1,\ell_1,\dotsc,\ell_k)$ such that
\begin{subequations}
\label{5pol}
\begin{align}
\label{5pol1}
\PP(\ee_1,\dotsc ,\ee_{k}, 1 )  &
  = \PP(\ee_1,\dotsc,\ee_k) \cap \{ \ell_{k+1} \geq \ell \} \\
\label{5pol2}
\PP(\ee_1,\dotsc ,\ee_{k}, -1 ) &
= \PP(\ee_1,\dotsc,\ee_k) \cap \{ \ell_{k+1} \leq \ell \} .
\end{align} 
\end{subequations}
Let us prove \eqref{5pol} by induction on $k$. When $k=0$ we have $\PP(1) = F + \R_+ v = \{ \ell_1 \geq x_1\}$ and $\PP(-1)=\{\ell_1\leq x_1\}$. If $k\geq 1$, let $v$ be the normalized ($v_1=1$) axis of $\PP$ and assume inductively that 
\[ \PP_{\ee_1}(\ee_2,\dotsc,\ee_k1) = 
\PP_{\ee_1}(\ee_2,\dotsc,\ee_k) \cap \{  \ell_{k+1} \geq \ell \} \] 
for some $\ell \in \lin( 1,\ell_2 , \dotsc ,\ell_k)$. Let us assume that $\ee_1=1$, then 
\begin{align*}
\PP( 1 , \ee_2 , \dotsc , \ee_k , 1) & = 
 ( \pi_{F,v} )^{-1}
  \bigl(  \PP_1(\ee_2,\dotsc,\ee_k) \cap  \{ \ell_{k+1} \geq  \ell \}  \bigr) \\
 & =   \PP(1,\ee_2,\dotsc,\ee_k) \cap A_v ,
\end{align*}
where $A= F \cap \{ \ell_{k+1} \geq  \ell \}$. Using \eqref{5proj}, we obtain
\[ A_v = \{ x_1\geq \alpha \} \cap \{ \ell_{k+1} \geq \ell' \} \]
where $\ell' = \ell + (\ell_1 - \alpha) ( v_{k+1} - \vec{\ell}(v))$. Since additionnally $\{x_1\geq \alpha\}$ contains $\PP(1,\ee_2,\dotsc,\ee_k)$, we get \eqref{5pol1} when $\ee_1=1$. The proof for $\ee_1=-1$ and the proof of \eqref{5pol2} are similar. \\
From \eqref{5pol}, we get in particular
\[ \PP(\ee_1,\dotsc,\ee_k) = \PP(\ee_1,\dotsc,\ee_k,1) \cup \PP ( \ee_1,\dotsc,\ee_k,-1). \]
Applying this equality again and again, we obtain, for any $\ee$ of size $k$
\[ \PP(\ee) = \cup \bigl\{ \PP(\ee,\tau) {\ \vert \ } \tau \in \{-1,1\}^{n-k} \bigr\} \]
where $(\ee,\tau)$ is the sequence obtained by concatenation of $\ee$ and $\tau$. Therefore, if $\PP$ is an equipartition for $\mu$ then
\begin{equation}
\label{5equipa}
 \mu \bigl( \PP(\ee_1,\dotsc,\ee_k) \bigr) = 2^{- k } \mu (E) .
\end{equation}
\section{Uniqueness}
In this section we prove that, under reasonable assumptions, the Yao-Yao center of a measure, with respect to a given basis, is unique. In the sequel, the space $E$ is equipped with a system of coordinates $\ell_1,\dotsc ,\ell_n$ and all Yao-Yao partitions are adapted to this system.
\begin{Lem}
\label{5Av}
Let $\alpha\in \R$ and $F=\{ z \in E \, | \, z_1=\alpha \}$. 
\begin{enumerate}
\item[(i)] Let $A\subset F$ depend only on $\ell_2,\dotsc,\ell_k$, let $v,w\in\vec{E}$ satisfying $v_1=w_1=1$ and $(v_2,\dotsc, v_k)=(w_2,\dotsc, w_k)$, then $ A_v = A_w $ .
\item[(ii)] Let $\ell$ be an affine form on $E$, non-constant on $F$ and let $A = F\cap \{ \ell \geq 0\}$. Let $v,w$ satisfy $v_1=w_1=1$ and $\vec{\ell} (v) > \vec{\ell}(w)$, then $A_v \subsetneq A_w$.
\item[(iii)] Again let $A = F\cap \{ \ell \geq 0\}$, and let $(v^p)$ be a sequence satisfying $v^p_1 = 1$ for all $p$ and $\vec{\ell}(v^p)\uparrow + \infty$. Then for any $\mu \in \MM(E)$ we have 
 \[ \lim_{p\rightarrow +\infty} \mu ( A_{v^p} ) = 0  . \]
\end{enumerate}
\end{Lem}
\begin{proof}
Point \textit{(i)} follows easily from \eqref{5proj}. For \textit{(ii)}, observe that when $A=F\cap\{\ell\geq 0\}$, equation~\eqref{5proj} becomes
\begin{equation}
 \label{5Avbis}
A_v  =  \{ x\in E  \, | \, x_1\geq \alpha \; \text{and} \; \ell(x) \geq (x_1-\alpha) \vec{\ell}(v) \} ,
\end{equation}
and similarly for $A_w$. The inclusion $A_v\subset A_w$ follows immediately. Besides, since $\ell$ is non-constant on $F$, we can find $x$ satisfying $x_1=1+\alpha$ and
\[ \vec{\ell}(w) < \ell(x) < \vec{\ell}(v) , \]
so inclusion is strict. By \textit{(ii)} the sequence $(A_{v^p})$ is decreasing, and clearly by \eqref{5Avbis} the intersection $\cap A_{v^p}$ is included in $F$, hence \textit{(iii)}.
\end{proof}
The key step is the following
\begin{Lem}
\label{5inclusion}
Let $\PP$ and $\QQ$ satisfy 
\begin{align}
 \label{5hyp1}
& \forall \ee_1,\dotsc,\ee_k, \quad \PP(\ee_1,\dotsc,\ee_k)=\QQ(\ee_1,\dotsc,\ee_k) \\
\label{5hyp2}
& \exists \ee'_1,\dotsc,\ee'_{k+1}, \quad \PP(\ee'_1,\dotsc,\ee'_{k+1})
              \neq \QQ(\ee'_1,\dotsc,\ee'_{k+1})
\end{align}
for some $k<n$. If $x_{k+1}\geq y_{k+1}$, then there exists $\delta_1,\dotsc,\delta_k$ such that $\PP(\delta_1,\dotsc,\delta_k ,1)$ is strictly included in $\QQ(\delta_1,\dotsc,\delta_k ,1)$.
\end{Lem}
\begin{proof}
The proof is by induction on $k$.
If $k=0$ then \eqref{5hyp2} becomes
\[ \PP(1) = \{\ell_1\geq x_1\} \neq  \{\ell_1\geq y_1\} = \QQ(1). \]
So $x_1\neq y_1$, thus $x_1 > y_1$, hence $\PP(1) \subsetneq \QQ(1)$.  \\
Assume that $k\geq 1$. Let $v$ and $w$ be the normalized ($v_1=w_1=1$) axes of $\PP$ and $\QQ$, respectively. Recall that \eqref{5hyp1} implies \eqref{5center-axis}: $(x_1,\dotsc, x_k)= (y_1,\dotsc ,y_k)$ and $(v_1,\dotsc, v_k)= (w_1,\dotsc, w_k)$. Also, intersecting \eqref{5hyp1} with $\{\ell_1=x_1\}$, we get
\begin{equation}
 \label{5hyp1bis}
\forall \ee_1,\dots\ee_k , \quad \PP_{\ee_1} (\ee_2,\dotsc,\ee_k) =
         \QQ_{\ee_1} (\ee_2,\dotsc,\ee_k) .
\end{equation} 
There are three cases.
If $v_{k+1}=w_{k+1}$, since for all $\ee_1,\dotsc,\ee_{k+1}$, the set $A = \PP_{\ee_1} (\ee_2,\dotsc,\ee_{k+1})$ depends only on $\ell_2,\dotsc, \ell_{k+1}$, Lemma~\ref{5Av}~\textit{(i)} implies that
\begin{equation}
 \label{5etape1}
\begin{split}
\PP (\ee_1,\dotsc,\ee_{k+1}) & =
\PP_{\ee_1} (\ee_2,\dotsc,\ee_{k+1}) + \R_+ (\ee_1v)  \\ 
 & = \PP_{\ee_1} (\ee_2,\dotsc,\ee_{k+1}) + \R_+ (\ee_1w).
\end{split}
\end{equation}
Also $\QQ(\ee_1,\dotsc,\ee_{k+1}) = \QQ_{\ee_1} (\ee_2,\dotsc,\ee_{k+1}) + \R_+ (\ee_1w)$.  
Therefore there must exist $(\delta_1,\ee'_2 ,\dotsc ,\ee'_{k+1})$ such that 
\[ \PP_{\delta_1}(\ee'_2,\dotsc,\ee'_{k+1})\neq \QQ_{\delta_1}(\ee'_2,\dotsc,\ee'_{k+1}) , \]
otherwise \eqref{5hyp2} would fail. Recalling \eqref{5hyp1bis}, we remark that we can apply the induction assumption to $\PP_{\delta_1}$ and $\QQ_{\delta_1}$: there exists $\delta_2,\dotsc,\delta_k$ such that $\PP_{\delta_1} (\delta_2,\dotsc,\delta_k 1)$ is stricly included in $\QQ_{\delta_1} (\delta_2,\dotsc,\delta_k 1)$. Then \eqref{5etape1} shows that  $\PP(\delta_1,\dotsc,\delta_k 1)$ is strictly included in $\QQ(\delta_1,\dotsc,\delta_k 1)$, which concludes the proof in this case. \\
If $v_{k+1}>w_{k+1}$, then by \eqref{5pol} and Lemma~\ref{5Av}~\textit{(ii)} we obtain
\begin{equation}
 \label{5etape2}
\PP_{1} (\ee_2,\dotsc,\ee_k 1) + \R_+ v  \subsetneq
       \PP_{1} (\ee_2,\dotsc,\ee_k 1) + \R_+ w ,
\end{equation}
for all $\ee_2,\dotsc,\ee_k$.
Therefore, it is enough to prove that there exists $\delta_2,\dotsc,\delta_k$ such that 
\[ \PP_{1}(\delta_2,\dotsc,\delta_k 1)\subset \QQ_{1}(\delta_2,\dotsc,\delta_k 1) . \]
This holds by the induction assumption applied to $\PP_1$ and $\QQ_1$: either we have equality for all $\delta_2,\dotsc,\delta_k $, or there exists $\delta_2,\dotsc,\delta_k $ such that the inclusion is strict. \\
The case $v_{k+1}<w_{k+1}$ is similar, we just need to deal with $\PP_{-1}$ and $\QQ_{-1}$ instead of $\PP_1$ and $\QQ_1$.
\end{proof}
We now let $\MM^*(E)$ be the set of measures $\mu\in\MM(E)$ which satisfy $\mu(U) >0$ for every open set $U$. Here is the main result of this section.
\begin{Prop}
\label{5unique}
Let $\mu\in\MM^* (E)$, then $\mu$ has at most one center with respect to the system $\LL$. Moreover the $k$ first coordinates of the center of $\mu$ depend only on the restriction of $\mu$ to $\BB_E(\ell_1,\dotsc,\ell_k)$.
\end{Prop}
\begin{proof}
Let $\mu,\nu \in\MM^* (E)$ satisfy $\nu(A) = \mu (A)$ for all $A \in \BB_E(\ell_1,\dotsc, \ell_l)$. We can assume that $\mu(E)=1(=\nu(E))$. Let $\PP$ and $\QQ$ be equipartitions for $\mu$ and $\nu$, respectively, and let us call $x$ and $y$ the respective centers of $\PP$ and $\QQ$. We want to prove that $(x_1,\dotsc ,x_l) = (y_1,\dotsc, y_l)$. If $\PP\neq \QQ$, there exists $k$ satisfying \eqref{5hyp1} and \eqref{5hyp2}, we have in particular 
\begin{equation}
 \label{5truc}
(x_1,\dotsc, x_k)=(y_1,\dotsc, y_k)
\end{equation}
and we may assume that $x_{k+1}\geq y_{k+1}$. By the previous lemma, there exists  $(\ee_1,\dotsc,\ee_{k+1})$ such that $\PP(\ee_1,\dotsc,\ee_{k+1})$ is strictly included in $\QQ(\ee_1,\dotsc,\ee_{k+1})$. Since these sets are polytopes, their difference $\QQ(\ee_1,\dotsc,\ee_{k+1})\backslash \PP(\ee_1,\dotsc,\ee_{k+1})$ has non-empty interior. Recall that the set $\QQ(\ee_1,\dotsc,\ee_{k+1})$ depends only on $\ell_1,\dotsc, \ell_{k+1}$. So if $k+1\leq l$, applying \eqref{5equipa}, and the fact that $\mu(U)>0$ for any open set $U$, we get
\[ 
\begin{split}
 2^{-k-1} = \mu \bigl( \PP(\ee_1,\dotsc,\ee_{k+1}) \bigr) & <
  \mu \bigl( \QQ(\ee_1,\dotsc,\ee_{k+1}) \bigr) \\ 
 & = \nu \bigl( \QQ(\ee_1,\dotsc,\ee_{k+1}) \bigr) = 2^{-k-1} , 
\end{split}
\]
which is absurd. Therefore $l \leq k$, and the result follows from \eqref{5truc}. \\
The uniqueness of the center is obtained by letting $\nu=\mu$ and $l=n$.
\end{proof}
\section{Continuity}
We now deal with the continuity of the center, for this we need a topology on $\MM(E)$. Let $(\mu^p)_{p\in\N}$ be a sequence of elements of $\MM(E)$ and $\mu\in \MM(E)$. The sequence $(\mu^p)$ converges \emph{narrowly} to $\mu$ when 
\begin{equation}
 \label{5narrowconv}
\int \phi \, d\mu^p \rightarrow \int \phi \, d\mu 
\end{equation}
for any function $\phi$ continuous and bounded on $E$. A subset $\FF$ of $\MM(E)$ is \emph{tight} if for every $\ee>0$ there exists a compact subset $K$ of $E$ such that $\mu(E\backslash K) < \ee$ for all $\mu \in \FF$. A converging sequence of measures is obviously tight. \\
Recall that when $T\colon\Omega_1\rightarrow \Omega_2$ is a measurable map, and $\mu$ is a measure on $\Omega_1$, the image measure $T_\# \mu$ of $\mu$ by $T$ is defined by
\[ T_\# \mu (A) = \mu \bigl( T^{-1} (A) \bigr) ,\]
for every measurable subset $A$ of $\Omega_2$. \\
Let $F$ be an affine hyperplane of $E$. For any $\mu\in\MM(E)$ and $v\in \vec{E} \backslash \vec{F}$ we let $\mu_{F,v}= (\pi_{F,v})_\# \mu$, hence
\begin{equation}
 \label{5muv}
 \forall A\subset F, \quad \mu_{F,v} (A) = \mu( A + \R_+ v) .
\end{equation}
Let $(\nu^p)$ be a sequence of elements of $\MM(E)$ converging narrowly towards 
$\nu\in\MM(E)$. Notice that if $A\subset E$ is a polytope, then $\nu(\partial A)=0$ and thus $\nu^p (A) \rightarrow \nu (A)$. \\
Let $(v^p)$ be a sequence of elements of $\vec{E}\backslash\vec{F}$ converging to $v \in \vec{E}\backslash\vec{F}$. Then
\[ \nu^p_{F,v^p} \rightarrow  \nu_{F,v}, \quad \text{narrowly}. \]
Indeed, by tightness of the set $\{\nu,\nu^p {\ \vert \ } p\in\N\}$, we can assume that all these measures are supported by a compact set $K$. We can also assume that $F + \R_+ v^p = F+ \R_+ v$ for all $p$. Then, setting $\mu^p = \nu^p_{F,v^p}$ and $\mu=\nu_{F,v}$, it is easy to see that the supports of $\mu$ and $\mu^p$ for all $p$ are contained in a compact subset $L$ of $F$. Let $\phi$ be continuous on $F$, then $\phi$ is uniformly continuous on $L$, from which we get $\phi \circ \pi_{F,v^p} \rightarrow \phi \circ \pi_{F,v}$, uniformly
on $K$. The result follows easily. \\
In the same spirit, if $(u^p)$ is a sequence of elements of $\vec{E}$ converging to $0$ and if $T_p$ is the translation of vector $u^p$, then $(T_p)_\# \nu^p \rightarrow \nu$, narrowly. \\
The following lemma is an easy consequence of Proposition~\ref{5convex}.
\begin{Lem}
 \label{5tightness}
Let $\mathcal{F}\subset \MM(E)$ be tight and such that $\{ \mu(E) {\ \vert \ } \mu \in \mathcal{F} \}$ is bounded. Then there exists a compact subset of $E$ containing any center of any element of $\mathcal{F}$.
\end{Lem}
We now reformulate the definition of a Yao-Yao center. Let $\mu\in\MM(E)$. Then an element $x\in E$ is a center for $\mu$ according to $(\ell_1,\dotsc,\ell_n)$ if and only if letting $F=\{ \ell_1 = x_1\}$ there exists $v\in\vec{E}\backslash\vec{F}$ such that
\begin{itemize}
\item[-] $ \mu (F+ \R_+ v) =  \tfrac{1}{2} \mu(E)$. 
\item[-] the point $x$ is a center for both $\mu_{F,v}$ and $\mu_{F,-v}$, according to the system of coordinates $({\ell_2}_{|F} ,\dotsc, {\ell_n}_{|F})$.
\end{itemize}
In the sequel, such a vector $v$, is called an axis for $(\mu,F)$, and we say that $v$ is normalized if $v_1=1$. 
Here is a first continuity property of the center.
\begin{Lem}
\label{5bidule}
Let $(\mu^p)$ be a sequence of elements of $\MM(E)$ converging narrowly towards $\mu\in\MM(E)$. If all the measures $\mu^p$ share a common center $x$ with respect to $(\ell_1,\dotsc, \ell_n)$; then $x$ is also a center of $\mu$.
\end{Lem}
\begin{proof}
The proof is by induction on the dimension $n$ of $E$. When $n=1$, the result is obvious: if $x$ is a median for all the measures $\mu^p$ then it is a median for $\mu$. We assume that $n\geq 2$ and that the result holds for any affine space of dimension $n-1$. We can also assume that $\mu$ and the measures $\mu^p$ are probability measures. Let $F=\{\ell_1 = x_1 \}$. For all $p$ there exists a normalized axis $v^p$ for $(\mu^p,F)$. We claim that the sequence $(v^p)$ is bounded. Indeed otherwise there exists $\ell$ affine on $E$ such that $\vec{\ell}(v^p) \rightarrow +\infty$. Let $H = \{ \ell \geq 0\}\cap F$. Let $\epsilon >0$, by Lemma~\ref{5Av}~\textit{(iii)}, there exists $w$ such that $w_1=1$ and $\mu(H_w)  < \epsilon$. For $p$ big enough $\vec{\ell}(v^p) > \vec{\ell}(w)$, applying Lemma~\ref{5Av}~\textit{(ii)} we get $H_{v^p} \subset H_w$. Also for $p$ large, we have $\mu^p(H_w) \leq \mu(H_w)+ \epsilon$. Thus
\[ \mu^p_{F,v^p} (H) = \mu^p ( H_{v^p} ) < 2 \epsilon  . \]
Taking $\epsilon$ small enough, it follows from Proposition~\ref{5convex} that for any such $p$ the center $x$ of $\mu^p_{F,v^p}$ does not belong to $H$. Hence $\ell(x) \leq 0$. Let $m>0$, the same holds if we replace $\ell$ by $\ell+m$, so we get $\ell(x)\leq -m$ for all $m>0$, which is absurd. \\
Up to an extraction we can assume that $(v^p)$ has a limit, say $v$ (which satisfies $v_1=1$). Then 
\[\mu^p_{F,v^p} \rightarrow \mu_{F,v} \quad \text{narrowly} . \]
By the induction assumption, $x$ is a center for $\mu_{F,v}$ with respect to the basis $( {\ell_2}_{|F} ,\dotsc, {\ell_n}_{|F} )$, and the same holds for $\mu_{F,-v}$. Therefore $x$ is a center for $\mu$.
\end{proof}
\begin{Cor}
\label{5continuity}
Let $(\mu^p)$ be a sequence of elements of $\MM(E)$ converging narrowly towards $\mu\in\MM(E)$. If every measure $\mu^p$ has a center $x_p$ with respect to $(\ell_1,\dotsc, \ell_n)$ and if the sequence $(x^p)$ has a limit $x$, then $x$ is a center of $\mu$ with respect to $(\ell_1,\dotsc, \ell_n)$.
\end{Cor}
\begin{proof}
Let $v^p = x-x^p$ and $T_p$ be the translation of vector $v^p$. Since $v^p\rightarrow 0$ we have $(T_p)_\# \mu^p \rightarrow \mu$, and clearly $x = T_p (x^p)$ is a center for $(T_p)_\# \mu^p$. Then the result follows from the previous lemma.
\end{proof}
\section{Proof of Theorem~\ref{5main}}
Let us start with an easy fact. 
Let $V$ be a vector space of finite dimension $n$, with a given basis $(e_1,\dotsc,e_n)$. Let $T:V \rightarrow V$ be continuous and satisfy the following properties:
\begin{enumerate}
\item[(a)] For $k=1\dotsc n$ and $v\in V$, the $k$ first coordinates of $Tv$ depend only on the $k$ first coordinates of $v$.
\item[(b)] If $f$ is a linear form and $(v^p)$ a sequence satisfying $f(v^p)\rightarrow +\infty$ then $f(Tv^p)\rightarrow +\infty$.
\end{enumerate}
Then $T$ is onto.

Indeed, let $u$ in $V$ and $u_1,\dotsc, u_n$ be its coordinates. By (b) and the continuity of $T$, we can find $v_1\in \R$ such that the first coordinate of $T(v_1 e_1)$ is $u_1$. Then we find $v_2 \in \R$ such that the second coordinate of $T( v_1 e_1 + v_2 e_2 )$ is $u_2$, and by (a) the first coordinate of $T ( v_1 e_1 + v_2 e_2 )$ is still $u_1$. And so on. 
\begin{Prop}
Let $E$ be an affine space and $(\ell_1,\dotsc,\ell_n)$ be a system of coordinates.
Any element of $\MM^*(E)$ admits a unique center with respect to $(\ell_1,\dotsc, \ell_n)$.
\end{Prop}
\begin{proof} 
We have already proved the uniqueness of the center, we shall prove its existence by induction on the dimension. If $E$ has dimension $1$ then we just have to show that any $\mu\in \MM^*(E)$ has a median, which is clear. \\
If $n\geq 2$, we assume that the proposition holds for any affine space of dimension $n-1$. Let $\alpha\in\R$ satisfy $\mu\{ \ell_1 \geq \alpha \} = \tfrac{1}{2} \mu(E)$ and let $F=\{\ell_1=\alpha\}$. Let $u\in \vec{E}$ satisfy $u_1=1$. By the induction assumption, for all $v\in \vec{F}$ and $\ee\in\{-1,1\}$, the measure  $\mu_{F,\ee(u+v)}$ admits a unique center with respect to $(\ell_2,\dotsc,\ell_n)$ which we call $x^{(\ee)}(v)$.  If we can prove that there exists $v$ such that $x^{(1)}(v)=x^{(-1)}(v)$, then we are done. We define
\[ T : v\in \vec{F} \mapsto x^{(-1)}(v) - x^{(1)}(v) \in \vec{F} . \]
If a sequence $(v^p)$ goes to $v$, then $\mu_{F,u+v^p}$ converges narrowly to $\mu_{F,u+v}$. Then by Lemma~\ref{5tightness}, the sequence $\bigl( x^{(1)}(v^p) \bigr)_p $ is bounded, and by Corollary~\ref{5continuity}, any of its converging subsequences goes to the unique center of $\mu_{F,u+v}$. Therefore $ x^{(1)}(v^p)\rightarrow x^{(1)}(v)$, and similarly for $x^{(-1)}$, hence the continuity of $T$. \\
If $(v_2,\dotsc, v_k) = (w_2,\dotsc, w_k)$ then, by Lemma~\ref{5Av} \textit{(ii)}, we have $A_{u+v} = A_{u+w}$ for all $A \in \BB_F( \ell_2,\dotsc,\ell_k)$, hence $\mu_{F,u+v}(A) = \mu_{F,u+w}(A)$. By Proposition~\ref{5unique}, this implies that for $i=2,\dotsc, k$
\[ \ell_i \bigl( x^{(1)} ( v ) \bigr) \ =  \ell_i \bigl(  x^{(1)} ( w ) \bigr) , \]
and similarly for $x^{(-1)}$. Therefore $T$ satisfies (a). \\
Let $\ell$ be an affine form on $E$, and $(v^p)$ be a sequence of elements of $\vec{F}$ such that $\vec{\ell}(v^p)\rightarrow + \infty$. Then it follows from Lemma~\ref{5Av} \textit{(iii)} that for any $m>0$ and any big enough $p$ we get $\mu_{u+v^p} ( F \cap \{ \ell \geq -m\} ) < 2^{-n} \mu(F)$, which by Proposition~\ref{5convex} implies that $\ell( x^{(1)} ( v^p ) ) \leq -m$. Hence 
\[ \ell\bigl( x^{(1)} (v^p) \bigr) \rightarrow - \infty . \]
Similarly $\ell\bigl( x^{(-1)} (v^p) \bigr) \rightarrow + \infty$. Thus $T$ satisfies (b). Therefore $T$ is onto. There exists $v\in\vec{F}$ such that $Tv=0$, then $x^{(1)} (v)=x^{(-1)} (v)$ which concludes the proof.
\end{proof}
\begin{Lem}
If $\mu\in\MM^*(E)$ has a center of symmetry $z$ then $z$ is the unique Yao-Yao center of $\mu$, whatever the basis $\LL$.
\end{Lem}
\begin{proof}
Let $x$ be the Yao-Yao center of $\mu$ and $s: E\rightarrow E$ be the symmetry of center $z$. Then clearly $s (x)$ is a center for $s_{\#} \mu$, with respect to $\LL$. Since $s_{\#} \mu = \mu$ and by uniqueness of the center, we obtain $s(x)=x$. Therefore $x=z$. 
\end{proof} 
We are now in a position to prove Theorem~\ref{5main}. Let $E$ be an affine space of dimension $n$ and $(\ell_1,\dotsc,\ell_n)$ be a system of coordinates. Let $\gamma$ be an arbitrary element of $\MM^*(E)$. Let $\mu\in\MM(E)$ and for $p\geq 1$ let $\mu^p = \mu+ \frac{1}{p} \gamma$. Then obviously $\mu^p\in\MM^*(E)$ and $\mu^p\rightarrow \mu$, narrowly. Let $x^p$ be the center of $\mu^p$ with respect to $(\ell_1,\dotsc,\ell_n)$. By Lemma~\ref{5tightness} and Corollary~\ref{5continuity}, the sequence $(x^p)$ is bounded and the limit of any of its converging subsequence is a center for $\mu$, so $\mu$ has a center. \\
If $\mu$ is symmetric with respect to $x$, then we let $\gamma$ be an element of $\MM^*(E)$ symmetric with respect to $x$. Then so is $\mu^p$, by the preceding lemma we get $x^p=x$ for any $p$, therefore $x$ is a center for $\mu$.


\begin{thebibliography}{plain}
%
\bibitem{alon}
N.~Alon et al. Crossing patterns of semi-algebraic sets, J. Comb. Theory Ser. A 111 (2005) 310--326. 

\bibitem{fm}
M.~Fradelizi and M.~Meyer.
Some functional forms of Blaschke-Santal\'o inequality, Math. Z. 256 (2007) 379--395.

\bibitem{moi}
J. Lehec. 
Partitions and functional Santal«o inequalities, Arch. Math. 92 (1) (2009) 89--94.

\bibitem{yaoyao}
A.~C.~Yao and F.~F.~Yao. A general approach to $d$-dimensional geometric queries, in {\em Proceedings of the seventeenth annual ACM symposium on Theory of computing}, ACM Press (1985) 163--168.

\end{thebibliography}
\end{document}